\documentclass[a4paper,10pt]{article}
\usepackage{latexsym,amssymb,amsfonts,amsmath,amsthm,mathrsfs,amstext,color,graphicx,times}
\usepackage[mathscr]{euscript}
\usepackage{indentfirst}

\newcommand{\R}{\mathbb{R}}
\newcommand{\Q}{\mathbb{Q}}
\newcommand{\N}{\mathbb{N}}

\usepackage{wrapfig}
\usepackage[all]{xy}
\usepackage{tikz}
\usetikzlibrary{patterns}

\newtheorem{theorem}{Theorem}[section]
\newtheorem{corollary}[theorem]{Corollary}
\newtheorem{proposition}{Proposition}[section]
\newtheorem{remark}{Remark}
\newtheorem{lemma}{Lemma}[section]
\newtheorem{example}{Example}[section]
\newtheorem{definition}{Definition}[section]
\newcommand{\tos}{\rightrightarrows} 

\DeclareMathOperator*{\argmin}{arg\,min}
\DeclareMathOperator{\co}{co}
\DeclareMathOperator{\Fix}{Fix}
\DeclareMathOperator{\EP}{EP}
\DeclareMathOperator{\CFP}{MEP}
\DeclareMathOperator{\inte}{int}

\title{Quasi-Equilibrium Problems with Non-self Constraint Map}
\author{John Cotrina
	\thanks{Universidad del Pac\'ifico. Av. Salaverry 2020, Jes\'us Mar\'ia, Lima, Per\'u. Email: 
	\texttt{\{ cotrina\_je,~zuniga\_jj\}@up.edu.pe}} 
	\and Javier Z\'u\~niga\footnotemark[1]}

\begin{document}
\maketitle
\begin{abstract}
In 2016 Aussel, Sultana and Vetrivel developed the concept of projected solution for quasi-variational inequality
problems and projected Nash equilibrium. 
We introduce a new concept of solution for quasi-equilibrium problems and
we study the existence of such solutions. 
Additionally, as a consequence of our results, we give existence results of projected solutions for quasi-optimization problems, quasi-variational inequalities problems and generalized Nash equilibrium problems.
\bigskip

\noindent{\bf Keywords}:  Quasi-equilibrium problem,~Quasi-variational inequality,~Generalized Nash equilibrium,~Non-self map

\bigskip

\noindent{{\bf MSC (2010)}: 49J40,~90C26,~90B10 } 

\end{abstract}

\section{Introduction}
Given  a non-empty subset $C$ of $\R^n$ and a 
bifunction $f:\R^n\times \R^n\rightarrow \R$, the \emph{equilibrium problem} (EP) is the following:
\begin{align}\label{EP}
\mbox{find }x \in C \mbox{ such that }f (x, y) \geq 0,~\mbox{ for all } y \in C. \tag{EP}
\end{align}
The equilibrium problem was introduced in \cite{OB93} and has been extensively studied in recent years (see e.g. 
\cite{castellani2012,JCYG,IKS06,FFB} and the references therein). 
Related to \eqref{EP}, it is natural to consider the following problem:
\begin{align}\label{cfp}
\mbox{find }x\in C\mbox{ such that }f(y,x)\leq 0,\mbox{ for all }y\in C, \tag{MEP}
\end{align}
which was called \emph{Minty equilibrium problem} (MEP) in \cite{castellani2012}. This problem corresponds to a particular case of the \emph{Convex Feasibility Problem} \cite{KKM,Kfan2}. 
It was established in \cite{ACI}, that if $f$ has the upper sign property, 
then every solution of \eqref{cfp} is a solution of \eqref{EP}. Moreover, both solution sets trivially coincide when $f$ is also pseudomonotone.

The classical example of equilibrium problem is the variational inequality problem (see e.g. \cite{John2001,AH04}), which is defined as follows:
a Stampacchia variational inequality problem is formulated as
\begin{align*}
\left.\begin{array}{l}
      \mbox{find }x\in C\mbox{ such that there exists }x^*\in T(x)\\
      \mbox{with }\langle x^*,y-x\rangle\geq0, \mbox{ for all }y\in C,
      \end{array}
\right. 
\end{align*}
where $T:\R^n\tos \R^n$ is a set-valued map and $\langle\cdot,\cdot\rangle$
denotes the Euclidean inner product. So, if $T$ has compact values, and  we define
the representative bifunction $f_T$ of $T$ by
\begin{align}\label{T-f}
 f_T(x,y)=\sup_{x^*\in T(x)}\langle x^*,y-x\rangle, 
\end{align}
it follows that every solution of the equilibrium problem associated to $f_T$ and $C$ is a solution of
the variational inequality problem associated to $T$ and $C$, and conversely.

Given a set-valued map $K:C\tos C$, the \emph{quasi-equilibrium problem} (QEP) associated to $f$ and $K$
is the following
\begin{align}\label{QEP}
 \mbox{find }x\in K(x) \mbox{ such that }  f(x,y)\geq0,~\mbox{for all }y\in K(x).\tag{QEP}
\end{align}
The associated \emph{Minty quasi-equilibrium problem} (MQEP), consists of
\begin{align}\label{MQEP}
\mbox{find }x\in K(x)\mbox{ such that }f(x,y)\leq0,~\mbox{for all }y\in K(x).\tag{MQEP}
\end{align} 

A (Minty) quasi-equilibrium problem is an (Minty) equilibrium problem in which the constraint set depends on the optimizing variable.
This dependence allows one to model some complex problems such as quasi-optimization problems, quasi-variational inequalities,
generalized Nash equilibrium problems, among others. These problems are unified in a convenient way,
and many of the results obtained for one can be extended, with suitable modifications, to 
general quasi-equilibrium problems, thus obtaining wider applicability. 
 
A quasi-equilibrium problem is defined by a bifunction and a  constraint set-valued map.
In most of the results on the existence of solutions for quasi-equilibrium problems in the literature,
the constraint map is assumed to be a self-map (see for instance \cite{ACI,castellani_Giuli15,JC-JZ2,castellani2018}). Our aim in this paper is to study
quasi-equilibrium problems with non-self constraint map. This type of problems arises, for example, in the electricity market as in the work of Aussel, Sultana and Vetrivel in \cite{ASV-2016}. In this case there is usually no solution
to the quasi-equilibrium problem. 

We study the same concept of projected solution introduced in  \cite{ASV-2016}, but for quasi-equilibrium problems. In doing so we improve two general results presented in \cite{ASV-2016}. As a consequence of our results, we obtain applications for quasi-optimization problems, quasi-variational inequalities and Nash equilibria.

The paper is organized as follows. 
Notation and basic definitions are given in  Section \ref{preli}. In Section \ref{characterization}, we show that the notions of generalized monotonicity for bifunctions can be characterized in terms of solution sets of 
\eqref{EP} and \eqref{cfp}. Moreover, we show that the concept of pseudomonotonicity and upper sign property are related under suitable assumptions.
Then, in Section \ref{PS-QEP}, we present the projected solution for quasi-equilibrium problems and prove different results on the existence of such solutions and recover several well-know theorems, one of them is due to Ky Fan.
 Finally, in Section \ref{applications}, we consider three applications: first, we obtain an existence result for quasi-optimization
problems with a non-self constraint map; second, an application to quasi-variational inequalities is given; and finally, we show the existence of projected  solutions for Nash equilibria.

\section{Basic definitions and preliminaries}\label{preli}
Let $X$ and $Y$ be Hausdorff topological spaces and let $K:X\tos Y$ be a set-valued map.
We recall that $K$ is:
\begin{itemize}
 \item \emph{closed}, when for any net $(x_i,y_i)_{i\in I}$ in the graph of $K$ such that $(x_i,y_i)_{i\in I}$ converges to $(x_0,y_0)$, 
 we have $y_0\in K(x_0)$;
 \item \emph{lower semicontinuous by nets},  when for any $x_0$, and any net $(x_i)_{i\in I}$ converging to $x_0$ and any 
 $y_0\in K(x_0)$, there exists a subnet $(x_{\varphi(j)})_{j\in J}$ of $(x_i)$ and a net $(y_j)_{j\in J}$
 converging to $y_0$ such that $y_j\in K(x_{\varphi(j)})$, for all $j\in J$;
 \item \emph{lower semicontinuous by sets}, when for any $x_0$ and any neighborhood $V$ of $K(x_0)$, there exists a neighborhood $U$ of $x_0$ such that for all $x \in U$, the set $K(x) \cap V$ is not empty;
 \item \emph{upper semicontinuous (usc)}, when for any $x_0$ and any neighborhood $V$ of $K(x_0)$, there exists a 
 neighborhood $U$ of $x_0$ such that $K(U) \subset V$.
\end{itemize}
 
In \cite[Proposition 2.5.6]{VMPOS} the authors show that lower semicontinuity by nets is equivalent to lower semicontinuity by sets. Hence, from now on, we can use any of these two definitions interchangeably and refer to them as lsc. 
 
A fixed point of a set-valued map $T:X\tos X$ is a point $x\in X$ such that $x\in T(x)$. The set of fixed points of $T$ is denoted by $\Fix(T)$. 
 
Our existence result will be obtained as a consequence of Himmelberg's fixed point theorem, which is stated below and it can be found in \cite[Theorem 2]{Himmelberg}.

\begin{theorem}[Himmelberg]\label{HFPT}
 Let $A$ be a non-empty and convex subset of a Hausdorff, locally convex topological vector space $Y$,
 and let $T:A\tos A$ be a set-valued map. If $T$ is usc with convex, closed and non-empty values, and $T(A)$ is contained in some compact subset $N$ of $A$,
then $\Fix(T)$ is a non-empty set.
\end{theorem}  

We will also need the following selection theorem due to Michael which can be found in \cite[Theorem 3.1''']{Michael}.

\begin{theorem}[Michael]\label{MST}
Every lower semicontinuous set-valued map $\Phi$ from a metric space to $\R^n$ with non-empty and convex values admits a continuous selection. This means that there exists a continuous function $h$, with the same domain as $\Phi$, such that the graph of $h$ is included in the graph of  $\Phi$.
\end{theorem}
 
As a consequence of the two previous theorems, one can deduce the following result in a similar way to part of the proof of Theorem 2.1 in \cite{Cubio}.

\begin{corollary}\label{FPTLSC}
Given a non-empty, convex and closed subset $C$ of $\R^n$, if $\Phi:C \tos C$ is lsc with non-empty, convex values and $\Phi(C)$ is relatively compact, then $\Fix(\Phi)$ is a non-empty set.
\end{corollary}

We now recall some  different definitions of \emph{generalized monotonicity} (which we use throughout this article). 

A set-valued map $T:\R^n\tos \R^n$ is said to be:
\begin{itemize}
\item \emph{pseudomonotone} on a subset $C$ of $\R^n$ if, for all $x,y\in C$ and any $x^*\in T(x)$, $y^*\in T(y)$, the following implication holds
\[
\langle x^*,y-x\rangle\geq0~\Rightarrow~\langle y^*,y-x\rangle\geq0,
\]
\item \emph{quasimonotone} on a subset $C$ of $\R^n$ if, for all $x,y\in C$ and any $x^*\in T(x)$, $y^*\in T(y)$, the following implication holds 
\[
\langle x^*,y-x\rangle>0~\Rightarrow~\langle y^*,y-x\rangle\geq0,
\]
\item \emph{properly quasimonotone} on a convex subset $C$ of $\R^n$ if, for all $x_1,\dots,x_m\in C$ and $x\in\co(\{x_1,\dots,x_m\})$ (the convex hull), there exists $i$ such that 
\[
\langle x_i^*,x-x_i\rangle\leq0,~\forall x_i^*\in T(x_i).
\]
\end{itemize}
In a similar way, a given bifunction $f:\R^n\times \R^n\rightarrow \R$ is said to be:
\begin{itemize}
\item \emph{pseudomonotone} on a subset $C$ of $\R^n$ if, for all $x,y\in C$, the following implication holds
\[
f(x,y)\geq0\Rightarrow f(y,x)\leq0;
\]
\item \emph{quasimonotone} on a subset $C$ of $\R^n$ if, for all $x,y\in C$, the following implication holds
\[
f(x,y)>0\Rightarrow f(y,x)\leq0;
\]
\item \emph{properly quasimonotone} on a convex subset $C$ of $\R^n$ if, for all $x_1,\dots, x_m\in C$, and all 
$x\in \co(\{ x_1,\dots,x_m\} )$, there exists $i$ such that
\[
f(x_i,x)\leq0.
\]
\end{itemize}
When $C=\R^n$, we  only say that a set-valued map (or a bifunction) is pseudomonotone (quasimonotone or properly quasimonotone) instead of saying that it is pseudomontone (quasimonotone or properly quasimonotone) on $\R^n$. 

In the case of set-valued maps, pseudomonotonicity implies proper quasimonotonicity, which implies quasimonotonicity.
However, for bifunctions,  pseudomonotonicity implies proper quasimonotonicity, provided that the bifunction is quasiconvex with respect to its 
second argument (see \cite[Proposition 1.1]{BP01}). Moreover, no relationship exists between quasimonotonicity and 
proper quasimonotonicity of bifunctions (see the counter-examples in \cite{BP01}).

It is very well-known that a set-valued map $T:\R^n\tos \R^n$,  with non-empty and compact values, satisfies some generalized monotonicity if and only if, its bifunction $f_T$, defined as in \eqref{T-f}, does too. In  a similar spirit, we have the following result which is easy to check.

\begin{proposition}\label{S2-uno}
Let $T:\R^n\tos \R^n$ be a set-valued map with compact values. If $-T$ is pseudomonotone, then $-f_T$ is too.
\end{proposition}
\begin{remark}
The previous proposition is also true in Banach spaces. In this case we can use  weak$^*$-compactness instead of regular compactness.
\end{remark}
The converse of the previous result does not hold in general, as the following example shows.
\begin{example}
Let $T:\R\tos\R$ be a set-valued map defined by
\[
T(x)=\{-1,1\}, \mbox{ for all }x\in\R.
\]
Clearly, $-T$ is not pseudomonotone but $-f_T$ is pseudomonotone, because $f_T\geq0$ and it only vanishes on the diagonal of $\R\times\R$.
\end{example}

Another important concept is the upper sign condition, which is given first for set-valued maps and later for bifunctions. Let $C$ be a convex subset of $\R^n$. For a given $t \in \R$ and $x,y \in \R^n$, let $x_t=(1-t)x+ty$.
\begin{itemize}
\item  A set-valued map $T:\R^n\tos \R^n$ is said to be \emph{upper sign-continuous} on $C$ if, for all 
$x,y\in C$, 
the following implication holds
\begin{align*}
\left( \forall t\in]0,1[,~\inf_{x_t^*\in T(x_t)}\langle x_t^*,y-x\rangle\geq0\right)~\Rightarrow~\sup_{x^*\in T(x)}\langle x^*,y-x\rangle\geq0 .
\end{align*}

\item  A bifunction $f:\R^n\times \R^n\to\R$ is said to have  the 
\emph{upper sign property} on $C$ if, for all $x\in C$ and for every $y\in C$, 
the following implication holds
\begin{align*}
\bigl(
f(x_t,x)\leq0,~
\forall~ t\in\,]0,1[~
\bigr)
\Rightarrow ~ f(x,y) \geq0.
\end{align*}
\end{itemize}

Upper sign-continuity (\cite{H03}) is a very weak notion of continuity. For instance, any upper semicontinuous set-valued map is upper sign-continuous.
Moreover, any positive function on $\R$ is upper sign-continuous.  This notion plays an important role for proving the existence of
solutions of  variational inequalities and quasi-variational inequalities, see \cite{AH04,AC-2013}. In a similar spirit,
the upper sign property plays an important role in order to establish the existence of solutions of equilibrium problems and quasi-equilibrium problems, see \cite{castellani2012,ACI}.

For the sake of completeness, let us
recall also that a function $h:\R^n\to\R$  is said to be:
\begin{itemize}
 \item \emph{convex} if, for any $x,y\in \R^n$ and $t\in[0,1]$,
we have 
\[
h(x_t)\leq (1-t)h(x)+th(y);
\]
 \item \emph{quasiconvex} if, for any $x,y\in \R^n$ and $t\in[0,1]$,
we have
\[
h(x_t)\leq \max\{h(x),h(y)\}.
\]
\item \emph{semistrictly quasiconvex} if, it is quasiconvex and, for any $x,y\in \R^n$ such that $h(x)\neq h(y)$, 
the following holds
\[
h(x_t)< \max\{h(x),h(y)\}\mbox{, for all }t\in]0,1[.
\]
\end{itemize}
Clearly, every convex function is semistrictly quasiconvex. 
An equivalent and useful characterization of quasiconvexity is that the function $f$
is quasiconvex if and only if, its sublevel set $S_\lambda=\{x\in \R^n:h(x)\leq\lambda\}$ is convex, for all $\lambda\in\R$.

\section{Canonical relations }\label{characterization}

John, in \cite{John2001}, characterized the proper quasimonotonicity of set-valued maps by the non-emptiness of the solution set 
of Minty variational inequality problems associated to this set-valued map on compact sets. 
Bianchi and Pini established a similar result for bifunctions under lower semicontinuity and 
quasiconvexity, see \cite[Theorem 2.1]{BP01}.

In a similar way to  \cite[Theorem 2 and Corollary of Theorem 1]{John2001}, the next result
characterizes quasimonotonicity and pseudomonotonicity. Denote by $\EP(f,C)$ and $\CFP(f,C)$ the solution sets of the equilibrium problem and Minty equilibrium problem, respectively.

\begin{proposition}\label{pseudo-EP-CFP}
Let $f:\R^n\times \R^n\to\R$ be a bifunction. Then, the following hold
\begin{enumerate}
\item $f$ is quasimonotone if and only if, $\CFP(f,\{x,y\})\neq\emptyset$, for all $x,y\in \R^n$.
\item $f$ is pseudomonotone if and only if, $\EP(f,C)\subset \CFP(f,C)$, for every subset $C$ of $\R^n$.
\item If $-f$ is pseudomonotone, then $\CFP(f,C)\subset\EP(f,C)$, for every subset $C$ of $\R^n$. The converse holds
provided that $f$ vanishes on the diagonal of $\R^n\times \R^n$.
\end{enumerate}
\end{proposition}
\begin{proof}
\begin{enumerate}
\item It follows from the fact that $f$ is not quasimonotone if and only if, there exists $x,y\in \R^n$ such that 
$ f(x,y)>0$ and $f(y,x)>0$, which is equivalent to $\CFP(f,\{x,y\})=\emptyset$.
\item It is a straightforward adaptation of \cite[Theorem 2]{John2001}.
\item Let $x\in \CFP(f,C)$,  that means $f(y,x)\leq0$, for all  $y\in C$.
By pseudomonotonicity of $-f$, we have $f(x,y)\geq0$. Hence, $x\in\EP(f,C)$.

Conversely, let $x,y\in \R^n$ such that $-f(x,y)\geq 0$. We take $C=\{x,y\}$
and since $f(y,y)=0$, we have  $y\in \CFP(f,C)$. Thus,
$f(y,x)\geq0$ or equivalently $-f(y,x)\leq0$. 
\end{enumerate}
\end{proof}

The following example says that in part 3 of Proposition \ref{pseudo-EP-CFP}, the reciprocal does not hold in general.
\begin{example}
The bifunction $f:\R\times\R\to\R$ defined as follows
\[
f(x,y)=\left\lbrace\begin{array}{cl}
-1,&\mbox{if }(x,y)=(0,1)\\
1,&\mbox{if }(x,y)=(0,0)\\
0,&\mbox{otherwise}
\end{array}\right.
\]
satisfies that $\CFP(f,C)\subset\EP(f,C)$, for every subset $C$ of $\R$. However, $-f$ is not pseudomonotone.
\end{example}

It was shown in \cite[Proposition 3.1]{ACI} that under the upper sign property, the inclusion in part 3 of Proposition \ref{pseudo-EP-CFP} holds. The next two propositions show that pseudomonotonicity and  the upper sign property are related under
suitable assumptions.

\begin{proposition}\label{pseudo-upper-sign}
Let $C$ be a convex subset of $\R^n$ and $f:\R^n\times \R^n\to\R$ be a bifunction such that
one of the following assumptions holds
\begin{enumerate}
 \item $f(\cdot,y)$ is lower semicontinuous, for all $y\in C$;
 \item $f(x,\cdot)$ is upper semicontinuous, for all $x\in C$;
 \item $f$ vanishes on the diagonal of $C$ and $-f(\cdot,y)$ is semistrictly quasiconvex, for all $y\in C$; 
 \item $f$ vanishes on the diagonal of $C$ and $f(x,\cdot)$ is semistrictly quasiconvex, for all $x\in C$.
\end{enumerate}
If $-f$ is pseudomonotone, then $f$ has the upper sign property on $C$.
\end{proposition}
\begin{proof}
Let $x$ and $y$ be two elements of $C$ such that 
\begin{align}\label{upper-condition}
 f(x_t,x)\leq0,\mbox{ for all }t\in]0,1[.
\end{align}
 
\begin{enumerate}
 \item If $f(\cdot,y)$ is lower semicontinuous, then $f(y,x)\leq0$. Thus, the result follows from the pseudomonotonicity of $-f$.
 \item Since $-f$ is pseudomonotone, condition \eqref{upper-condition} implies that $f(x,x_t)\geq0$, for any $t\in]0,1[$.
By upper semicontinuity of $f(x,\cdot)$ we deduce that $f(x,y)\geq0$.
\item If $f(x,y)<0$, then $f(y,x)>0$ which in turn implies $f(x_t,x)>0$ for all $t\in]0,1[$, due to semistric quasiconvexity of $-f(\cdot,x)$. However, this fact is a contradiction
with \eqref{upper-condition}. Hence, $f(x,y)\geq0$.
\item Suppose that $f$ does not have the upper sign property on $C$. Thus, there exist $x$ and $y$ in $C$ such that $f(x,y)<0$ and \eqref{upper-condition} holds. Since $f(x,x)=0$ and $f(x,\cdot)$ is semistrictly quasicovex, we have $f(x,x_t)<0$, for all $t\in]0,1[$. Now, by pseudomonotonicity of $-f$, we obtain $f(x_t,x)>0$, for all $t\in]0,1[$, which is a contradiction.

\end{enumerate}
\end{proof}
An important consequence of the previous result is the following corollary.
\begin{corollary} 
Let $T:\R^n\tos \R^n$ be a set-valued map with compact and non-empty values. If $-T$ is pseudomonotone, then it is upper sign-continuous.
\end{corollary}
\begin{proof}
It is enough to show that $f_T$, defined in \eqref{T-f}, has the upper sign property because from the definition of $f_T$, we have that $T$ is upper sign-continuous if and only if, $f_T$ has the upper sign property.
So, by Proposition \ref{S2-uno}, we have that $-f_T$ is pseudomonotone. Thus, $f_T$ has the upper sign property due to  part 4 of Proposition \ref{pseudo-upper-sign}.  
\end{proof}

\begin{remark}
The previous corollary is also true in Banach spaces. In that case, we can use weak$^*$-compactness instead of regular compactness.
\end{remark}

\begin{proposition}\label{upper-pseudo}
Let  $f:\R^n\times \R^n\to\R$ be a bifunction such that
the following assumptions hold
\begin{enumerate}
 \item $f$ vanishes on the diagonal of $\R^n\times \R^n$ and
 \item $f(\cdot,y)$ is quasiconvex for all $y\in \R^n$.
\end{enumerate}
If $f$ has the upper sign property on $\R^n$, then $-f$ is pseudomonotone.
\end{proposition}
\begin{proof}
Let $x$ and $y$ be two elements of $\R^n$ such that 
\begin{align}\label{no-pseudo}
f(x,y)\leq0~\mbox{ and }~ f(y,x)<0.
\end{align}
By quasiconvexity we obtain $f(x_t,y)\leq0$, for all $t\in]0,1[$. We now apply the upper sign property of $f$ and deduce that $f(y,x)\geq0$, which is a contradiction
with \eqref{no-pseudo}.
\end{proof}

\begin{remark} A few remarks about the previous results are given below. 
\begin{itemize}
\item A bifunction satisfying condition 3 in Proposition \ref{pseudo-upper-sign} is actually properly quasimonotone, due to \cite[Proposition 1.1]{BP01}.

\item In general, the upper sign property of $f$ and pseudomonotonicity of $-f$ are independents. Consider for instance   the bifunctions $f_1, f_2:\R\times\R\to\R$, defined by
\[
 f_1(x,y)=\left\lbrace\begin{array}{cl}
                     1,&\mbox{if }x\in\Q,~y\notin\Q\mbox{ or }x\notin\Q,~y\in\Q\\
                   0,&\mbox{if }x=y\\
                   -1,&\mbox{otherwise}
                  \end{array}
\right.
\]
and
\[
 f_2(x,y)=\left\lbrace\begin{array}{cl}
                     -x,&\mbox{if }y=0, x>0\\
                     y,&\mbox{if }x=0, y>0\\
                     0,&\mbox{otherwise}.
                    \end{array}
\right.
\]
Clearly, $f_1$ has the upper sign property and $-f_2$ is pseudomonotone,  but neither $-f_1$ is  pseudomonotone nor   $f_2$ has the upper sign property.

\item The bifunction $f_2$ also shows that the 
quasiconvexity assumption in Proposition \ref{upper-pseudo} can not be dropped.
\end{itemize}
\end{remark}

\section{Main results}\label{PS-QEP}

Any solution of \eqref{QEP} (or \eqref{MQEP}) will be called a ``\emph{classical solution}''.  
From now on, we denote the Euclidean norm of $\R^n$ by $\|\cdot\|$.

\begin{definition}
Given a non-empty subset $C$ of $\R^n$, a set-valued map $K:C\tos \R^n$ and  a 
bifunction $f:\R^n\times \R^n\rightarrow \R$, a point $x_0$ of $C$ is called a \emph{projected solution}
of:
\begin{itemize}
\item the QEP if, there exists $z_0\in\mbox{EP}(f,K(x_0))$ such that $x_0\in P_C(z_0)$,
\item the MQEP if, there exists $z_0\in\mbox{MEP}(f,K(x_0))$ such that $x_0\in P_C(z_0)$,
\end{itemize}
where $P_C$ denotes the projection onto $C$, that means
\[
P_C(z)=\{x\in C:~\|z-x\|\leq \|z-w\|\mbox{ for all }w\in C\}.
\]
\end{definition}

\begin{remark} \label{rempro}
It is clear that every classical solution is a projected solution. If additionally, $K$ is defined from $C$ to $C$, then the set of classical solutions is equal to the set of projected solutions.

On the other hand, we can see the projection as a set-valued map in general. Furthermore,
if $C$ is a convex, closed and non-empty subset of $\R^n$, then $P_C$
is a continuous function. 
\end{remark}


In a similar way to \cite{Cubio,castellani2018},
we now show the existence of projected solutions for quasi-equilibrium problems without upper semicontinuity of the constraint map by using Corollary \ref{FPTLSC}.
But before that, we need to introduce a few definitions.

Consider a non-empty set $C \subset \R^n$, a set-valued map $K:C\tos \R^n$, and a bifunction $f:\R^n\times\R^n\to\R$. Let  $Q:C\times \R^n\tos C\times \R^n$ be the set-valued map defined by \[ Q(x,z)=P_C(z)\times K(x), \] $F:\R^n\tos \R^n$ be defined by 
\[ F(z)=\{y\in \R^n:~f(z,y)<0\},  \] and $R: C \times \R^n \tos\R^n$ be defined by 
\[  R(x,z)=F(z)\cap K(x).\] 

The following lemma is not difficult to check and it establishes a characterization of projected solutions in terms of properties of the set-valued maps $Q$ and $R$.

\begin{lemma}\label{charac}
Let $C$ be a non-empty subset of $\R^n$, $K:C\tos \R^n$ be a set-valued map with non-empty values,
$f:\R^n\times\R^n\to\R$ be a bifunction, and $x\in C$. Then, $x$ is a projected solution of \eqref{QEP} if and only if, there exists $z\in\R^n$ such that
$(x,z)\in \Fix(Q)$ and $R(x,z)=\emptyset$.
\end{lemma}

We are now ready for our first existence result.

\begin{theorem}\label{Main-result-3}
Let $C$ be a non-empty, compact and convex subset of $\R^n$, $K:C\tos \R^n$ be a set-valued map and $f:\R^n\times\R^n\to\R$ be a bifunction. Assume that
\begin{enumerate}
\item $Q$ is lsc with non-empty convex values;
\item $Q(C\times\R^n)$ is relatively compact;
\item $\Fix(Q)$ is closed;
\item $R$ is lsc with convex values on $\Fix(Q)$;
\item $f(z,z)\geq0$, for every $z\in M$, where
\[
M= \{w\in K(C):~\mbox{there exists }u\in C\mbox{ such that }(u,w)\in \Fix(Q)\}.
\]
\end{enumerate}
Then, there exists a projected solution of \eqref{QEP}.
\end{theorem}
In order to prove the previous result we need the following lemma, which can be found in \cite[Lemma 2.3]{Naselli-Ricceri}.
\begin{lemma}\label{lemma-Naselli}
Let $X,Y$ two topological spaces and $A$ a closed subset of $X$. Consider two lsc set-valued maps $F:X\tos Y,\Phi:A\tos Y$ such that, for every $x\in A$, one has
$\Phi(x)\subset F(x)$. Let $G:X\tos Y$ be defined as
\[
G(x)=\left\lbrace\begin{array}{cl}
F(x),&\mbox{if }x\in X\setminus A\\
\Phi(x),&\mbox{if }x\in A.
\end{array}\right.
\]
Then, the set-valued map $G$ is lsc.
\end{lemma}
\begin{proof}[of Theorem \ref{Main-result-3}]
First notice that the non-emptyness of $\Fix(Q)$ is guaranteed by Corollary \ref{FPTLSC}.
Let $S:C\times \R^n\tos C\times\R^n$
be the set-valued map defined by
\[
S(x,z)=\left\lbrace\begin{array}{cl}
Q(x,z),&\mbox{if }(x,z)\in C\times\R^m\setminus \Fix(Q)\\
P_C(z)\times R(x,z),&\mbox{if }(x,z)\in \Fix(Q).
\end{array}\right.
\]
Since $S(x,z)\subset Q(x,z)$ for each $(x,z)\in C\times\R^n$, the lower semicontinuity of $S$ follows from Lemma \ref{lemma-Naselli} and the fact that $Q$ is lsc. Moreover, $S$ is convex valued. Since $Q(C\times\R^n)$ is relatively compact, so it is $S(C\times\R^n)$.  Corollary \ref{FPTLSC} implies that there exists $(x_0,z_0)\in C\times\R^n$ such that $(x_0,z_0) \in S(x_0,z_0)$.
This in turn implies $(x_0,z_0)\in P_C(z_0)\times R(x_0,z_0)$, that means $z_0\in R(x_0,z_0)$, but this a contradiction because $z_0 \in M$. Thus, there exists some $(x_0,z_0)\in C\times\R^n$ such that $S(x_0,z_0)=\emptyset$, which means that $(x_0,z_0)\in \Fix(Q)$ and $R(x_0,z_0)=\emptyset$. The
result follows then from  Lemma \ref{charac}.
\end{proof}

\begin{remark}\label{inf-fin}
Since $P_C$ is a continuous function,
$Q$ is lsc with non-empty and convex values, provided that
 the set-valued map $K$ is lsc with non-empty convex values.
On  the other hand, if the bifunction $f$ is quasiconvex with respect to its second argument,
 then $F$ and $R$ are convex valued. Moreover, if $f$ is continuous, then $F$ has open graph.
\end{remark}

In order to show the lower semicontinuity of $R$ we give sufficient conditions in the following result, which is inspired by \cite[Lemma 4.2]{Nicho}.

\begin{proposition}\label{lsc-int}
Let $X,~Y,~Z$ be topological spaces, and $T_1:X\tos Y$, $T_2:Z\tos Y$ be set-valued maps such that $T_1$ has open graph and $T_2$ is lsc. Then, the set-valued map $T:X\times Z\tos Y$, defined by
\[
T(x,z)=T_1(x)\cap T_2(z)
\] 
is lsc.
\end{proposition}
\begin{proof}
Let $V$ be an open subset of $Y$ and $(x_0,z_0)$ be an element of $X\times Z$ such that
$T(x_0,z_0)\cap V\neq\emptyset$. For $y_0\in T(x_0,z_0)\cap V$, since $T_1$ has open graph, we deduce that there exist $V_{x_0}$ and $V_{y_0}$, open subsets of $X$ and $Y$ respectively, such that $(x_0,y_0)\in V_{x_0}\times V_{y_0}$, where $V_{x_0}\times V_{y_0}$ is a subset of the graph of $T_1$. By the lower semicontinuity of $T_2$, there exists $V_{z_0}$, an open subset of $Z$, such that $z_0\in V_{z_0}$ and $T_2(z')\cap V_{y_0}\neq\emptyset$. Thus, taking $V_{x_0}\times V_{z_0}$ we  have that 
\[ T(x',z')\cap V=T_1(x')\cap T_2(z')\cap V\neq\emptyset. \]
\end{proof}

Thanks to Theorem \ref{Main-result-3}, Remark \ref{inf-fin} and Proposition \ref{lsc-int} we have the following corollary.

\begin{corollary}\label{P-result}
Assume that $C$ is compact, convex and non-empty.
If the following hold
\begin{enumerate}
\item $K$ is closed and lsc with convex values;
\item $K(C)$ is a compact subset of $\R^n$;
\item $f$ is continuous and quasiconvex with respect to its second argument;
\item $f$ vanishes on the diagonal of $\R^n\times \R^n$;
\end{enumerate}
then, there exists a projected solution of \eqref{QEP}.
\end{corollary}

We now present an alternative proof of the previous result, which does not follow from Theorem \ref{Main-result-3} but it also works in Banach spaces. 
However, we need to introduce first a few definitions.

The set-valued map $S:C\times\R^n\tos\R^n$ is defined by
\[
S(x,z)=\argmin_{y\in K(x)}f(z,y)
\]
and $T:C\times\R^n\tos C\times \R^n$ is defined by
\[
T(x,z)=P_C(z)\times S(x,z).
\]
It is clear that $T(C\times\R^n)\subset C\times \R^n$.

We also need the following two lemmas in order to establish the proof of Corollary \ref{P-result}. The first one characterizes projected solutions of \eqref{QEP} as fixed points of $T$.

\begin{lemma}\label{equiv}
Let $x_0\in C$ and assume that $f$ vanishes on the diagonal of $\R^n\times \R^n$. 
Then, $x_0$ is a projected solution of \eqref{QEP} if and only if, there exists
$z_0\in X$ such that $(x_0,z_0)$ is a fixed point of $T$.
\end{lemma}
\begin{proof}
Since $f(x,x)=0$ for all $x\in \R^n$, the result follows from the fact that for any $(x,z)\in C\times \R^n$,  
$z\in \mbox{EP}(f,K(x))$ if and only if,
$z\in S(x,z)$.
\end{proof}

\begin{remark} The equivalence in  Lemma \ref{equiv} does not hold if we assume $f$ to be positive on the diagonal of $\R^n\times \R^n$.
Consider for instance the bifunction $f:\R\times\R\to\R$ and the constraint set-valued map $K:[0,1]\tos\R$, both defined by
\[
f(x,y)=\left\lbrace\begin{array}{cl}
                    1,&\mbox{if  }x=y\\
                    0,&\mbox{if }x\neq y
                   \end{array}
 \right. \quad \mbox{and} \quad K(x)=[0,1+x].
\]
Clearly $K$ is closed and lsc. Also,
for each $z\in \R$ we have that $S(x,z)=K(x)\setminus\{z\}$, for all $x\in[0,1]$ and $z\in\R$. So,
 $T$ does not have fixed points.
However, $0$ is a projected solution of  \eqref{QEP} associated to $f$ and $K$.
\end{remark}

The second lemma is the following.
\begin{lemma}\label{uno}
Let $C$ be a closed and non-empty subset of $\R^n$ and $f$ be a continuous bifunction. If $K$ is  closed and lsc, then the set-valued map $S$ is closed.
\end{lemma}

\begin{proof}
Let $(x_\alpha,z_\alpha,w_\alpha)_{\alpha\in A}$ be a net in the graph of $S$ 
converging to $(x_0,z_0,w_0)$.
The closeness of $K$ implies $w_0\in K(x_0)$.
Since $K$ is lsc, for each $u\in K(x_0)$, there exists a subnet
$(x_{\varphi(\beta)})_{\beta\in B}$ and a net $(u_\beta)_{\beta\in A}$ converging 
to $u$ such that $u_\beta\in K(x_{\varphi(\beta)})$ and
$f(z_{\varphi(\beta)},w_{\varphi(\beta)})\leq f(z_{\varphi(\beta)}, u_\beta)$, 
for all $\beta\in B$.  
By continuity of $f$, we obtain $f(z_0,w_0)\leq f(z_0,u)$. 
Therefore, $w_0\in S(x_0,z_0)$.
\end{proof}

\begin{proof}[of Corollary \ref{P-result}]
By  Lemma \ref{uno} and the quasiconvexity in the second argument of $f$, we deduce that $S$ is closed with compact and convex values.  
Moreover, since $K(C)$ is  compact, $S$ must be usc. Since $P_C$ is continuous, we deduce that $T$ is usc with
compact, convex and non-empty values.
Hence, by  Theorem \ref{HFPT},
there exists a fixed point of $T$. The result follows from  Lemma \ref{equiv}.
\end{proof}

\begin{remark}
In Banach spaces, the projection $P_C$ is always usc with convex, compact and non-empty values, provided that $C$ is convex, compact and non-empty; and it is enough in order to obtain the upper semicontinuity of $T$ in the previous proof. 
\end{remark}

\begin{remark}
We note that \[ \Fix(Q)=\{(x,x)\in C\times C:~x\in \Fix(K)\},  \] 
when the constraint map is a self-map. Moreover, the set $M$ in Theorem \ref{Main-result-3} coincides with $\Fix(K)$, and the set-valued map $R$ restricted to $\Fix(Q)$ is
 \[
 R(x,x)=\{y\in K(x):~f(x,y)<0\}.
 \]
Moreover, if $K(x)=C$, for all $x\in C$, and $f$ is upper semicontinuous with respect to its first argument, then by Proposition \ref{lsc-int}, $R$ is lsc.
\end{remark}

As a corollary of  Theorem \ref{Main-result-3},  Proposition \ref{lsc-int} and Remark \ref{rempro}, we recover the following existence result due to Cubiotti \cite{Cubio} by considering $K$ to be a self-map.
\begin{corollary}\cite[Theorem 2.1]{Cubio}
Let $C$ be a non-empty compact convex subset of $\R^n$, $K:C\tos C$ be a set-valued map, and
$f:C\times C\to\R$ be a bifunction. Assume that
\begin{enumerate}
\item $K$ is lsc with non-empty convex values;
\item $\Fix(K)$ is closed;
\item the set $\{(x,y)\in C\times C:~f(x,y)\geq0\}$ is closed;
\item for each $x\in C$, $f(x,\cdot)$ is quasiconvex on $K(x)$;
\item for each $x\in \Fix(K)$, $f(x,x)\geq0$.
\end{enumerate}
Then, there exists a classical solution of \eqref{QEP}.
\end{corollary}

As another corollary of  Theorem \ref{Main-result-3}, we recover the famous minimax inequality due to Ky Fan on finite dimensional spaces, which can be found in \cite{Kfan}.
\begin{corollary}[Ky Fan]
Let $C$ be a  non-empty compact and convex subset of $\R^n$ and $f:C\times C\to\R$ be a bifunction. Assume that
\begin{enumerate}
\item $f$ is upper semicontinuous with respect to its first argument,
\item $f$ is quasiconvex with respect to its second argument, and
\item $f$  is not negative on the diagonal of $C\times C$.
\end{enumerate}
Then, there exists a solution of \eqref{EP}.
\end{corollary}

Now, we will establish the existence of projected solutions for Minty quasi-equilibrium problems. 
\begin{theorem}\label{main-result-4}
Let $C \subset \R^n$ be a non-empty, compact and convex set, 
 $K:C\tos\R^n$ be a set-valued map and $f:\R^n\times \R^n\to\R$ be a bifunction.
 If the following assumptions hold
 \begin{enumerate}
\item $Q$ is lsc with non-empty convex values;
 \item $Q(C\times\R^n)$ is relatively compact;
 \item $\Fix(Q)$ is closed;
 \item the set-valued map $G:\Fix(Q)\tos\R^n$, defined by
 \[
 G(x,z)=\{y\in K(x):~f(y,z)>0\}
 \]
 is lsc;
 \item $f$ is properly quasimonotone on   $\co(K(C))$; 
\end{enumerate}  
then, there exists a projected solution of \eqref{MQEP}. Moreover, 
the set of projected solutions of \eqref{QEP} is non-empty whether $f$ has the upper sign property or $-f$ is pseudomonotone.
\end{theorem}
In order to proof the previous result we need the following lemma, which can be found in \cite[Theorem 5.9]{RW}.
\begin{lemma}\label{lemma-RW}
Let $T:\R^n\tos\R^m$ be a set-valued map. If $T$ is lsc, then so is the set-valued map $\co(T):\R^n\tos\R^m$
defined by
\[
\co(T)(x)=\co(T(x)).
\]
\end{lemma}
\begin{proof}[of Theorem \ref{main-result-4}]
Due to Lemma \ref{lemma-RW}, the set-valued map $\co(G):\Fix(Q)\tos \R^n$, which is defined as
\[
\co(G)(x,z)=\co(G(x,z)),
\]
is lsc too. Thanks to Lemma \ref{lemma-Naselli}, the set-valued map $\Phi:C\times \R^n\tos C\times \R^n$, defined as
\[
\Phi(x,z)=\left\lbrace\begin{array}{cl}
Q(x,z),&\mbox{if }(x,z)\in C\times\R^n\setminus \Fix(Q)\\
P_C(z)\times \co(G)(x,z),&\mbox{if }(x,z)\in \Fix(Q)
\end{array}\right.
\]
 is lsc.
 If $G$ is non-empty valued then, by  Corollary \ref{FPTLSC}, there exists $(x,z)\in \Fix(Q)$ such that $x\in P_C(z)$ and $z\in \co(G)(x,z)$. That means that there exist $z_1,z_2,\dots, z_n$ belonging to $G(x,z)$ such that
 $z\in\co(\{z_1,z_2,\dots,z_n\})$. However, we get a contradiction with the proper quasimonotonicity of $f$, because $f(z_i,z)>0$ for all $i$.
Hence, there exists $(x,z)\in C\times\R^m$ such that $(x,z)\in \Fix(Q)$ and $G(x,z)=\emptyset$. Therefore, $x$ is a projected solution of \eqref{MQEP}.

Finally, the existence of projected solutions for \eqref{QEP} is due to 
\cite[Proposition 3.1]{ACI} when $f$ has the upper sign property and 
 part 3 of Proposition \ref{pseudo-EP-CFP} if $-f$ is pseudomonotone.
\end{proof}

As a consequence of the previous result and Proposition \ref{lsc-int} we obtain the following corollary.

\begin{corollary}\label{P-result-GM}
Assume that $C \subset \R^n$ is compact, convex and non-empty set.
If the following hold
\begin{enumerate}
\item $K$ is closed and lsc with convex values;
\item $K(C)$ is a compact subset of $\R^n$;
\item $f$ is properly quasimonotone;
\item $f$ is quasiconvex with respect to its second argument;
\item $\{(x,y)\in K(C)\times K(C):~f(x,y)\leq0\}$ is closed;
\end{enumerate}
then, there exists a projected solution of \eqref{MQEP}. Moreover, 
the set of projected solutions of \eqref{QEP} is non-empty, whether $f$ has the upper sign property or $-f$ is pseudomonotone.
\end{corollary}

As in Corollary \ref{P-result}, we will present an alternative proof of the previous corollary, which also works in Banach spaces. First, we define the following set-valued maps. Let $M:C\tos \R^n$ be defined by
\[
M(x)=\mbox{MEP}(f,K(x))
\] and let $T':C\times X\tos C\times X$ be defined by
\[
T'(x,z)=P_C(z)\times M(x).
\]

In a similar way to Lemma \ref{charac}, we characterize projected solutions of \eqref{MQEP} as fixed points of $T'$.

\begin{lemma}\label{fixedpoint}
A point $x_0 \in C$ is a projected solution of \eqref{MQEP} if and only if, there exists a $z_0 \in X$ such that $(x_0,z_0) \in \Fix (T')$.
\end{lemma}

Finally, we need the following lemma which is a very straightforward
adaptation of  \cite[Proposition 2.2]{JC-JZ2}.
\begin{lemma}\label{dos}
If $K$ is a closed and lower semicontinuous set-valued map, and
the set $\{(x,y)\in K(C)\times K(C):~f(x,y)\leq0\}$  is closed; then,
the set-valued map $M$ is closed. 
\end{lemma}

\begin{proof}[of Corollary \ref{P-result-GM}]
By \cite[Proposition 2.4]{ACI}, $M$ is non-empty valued.
Moreover, by Lemma \ref{dos} and the fact that $M(C)$ is relatively compact, we deduce the upper semicontinuity of $M$. It is also clear that $T'$ is usc with convex and closed values. As $T'(C\times X)\subset C\times K(C)$, by Theorem \ref{HFPT}, $T'$ admits at least one fixed point. Therefore, $x_0$ is a projected solution of \eqref{MQEP} by Lemma \ref{fixedpoint}.

Finally, the existence of projected solution for \eqref{QEP} is due to 
\cite[Proposition 3.1]{ACI} when $f$ has the upper sign property and 
part 3 of Proposition \ref{pseudo-EP-CFP} if $-f$ is pseudomonotone.
\end{proof}

As a consequence of Theorem \ref{main-result-4} we recover the following result.
 
\begin{corollary}\cite[Theorem 4.5]{ACI}\label{T-ACI}
Let $f:\R^n\times\R^n\to\R$ be a bifunction, $C$ be a convex, compact and non-empty subset of $\R^n$,
and $K:C\tos C$ be a set-valued map. Suppose that the following properties hold
\begin{enumerate}
\item the map $K$ is closed and lsc with convex values, and $\inte(K(x))\neq\emptyset$, for all $x\in C$;
\item $f$ is properly quasimonotone;
\item $f$ is semistrictly quasiconvex and lower semicontinuous with respect to its second argument;
\item for all $x,y\in\R^n$ and all sequence $(y_k)_k\subset\R^n$ converging to $y$, the following implication holds
\[
\liminf_{k\to+\infty}f(y_k,x)\leq0~\Rightarrow~ f(y,x)\leq0,
\]
\item $f$ has the upper sign property.
\end{enumerate}
Then, the quasi-equilibrium problem admits a classical solution.
\end{corollary}

\begin{proof}
Let the set-valued map $G:\Fix(K)\tos\R^n$ be defined by \[G(x)=\{y\in K(x):~f(y,x)>0\}. \] It is enough to show that $G$ is lower semicontinuous.  
Let $V$ be an open subset such that $G(x)\cap V\neq \emptyset$ and let $y\in G(x)\cap V$. Since $\inte(K(x))\neq\emptyset$, there exists $w\in \inte(K(x))\cap V$ such that $f(w,x)>0$, due to assumption 4. Now, by the lower semicontinuity of $f$ in its second variable, there exists $U$, an open subset, with $x\in U$ such that
$f(w,x')>0$, for any $x'\in U$, that means $V\cap G(x')\neq\emptyset$. 
\end{proof}

As another direct consequence of  Theorem \ref{main-result-4}, we have an existence result for Minty equilibrium problems.
\begin{corollary}\label{coro-MEP}
Let $C \subset \R^n$ be a convex, compact and non-empty subset. Let also $f:C\times C\to\R$ be a properly quasimonotone bifunction. If the set-valued map $F:C\tos C$, defined as
\[
F(y)=\{x\in C:~f(x,y)>0\}
\]
is lsc, then there exists a solution of \eqref{cfp}. 
\end{corollary} 

We finish this section with  Theorem \ref{KKM}, which is inspired by \cite[Lemma 1]{Kfan2}.

Let $C$ be a non-empty and convex subset of $\R^n$.
A set-valued map $T:C\tos C$ is said to be a \emph{KKM map} if, for any $x_1,x_2,\dots,x_m\in C$ the following holds
\[
\co(\{x_1,x_2,\dots,x_m\})\subset \bigcup_{i=1}^m T(x_i).
\] 

Given a set-valued map $T:C\tos C$, we define the set-valued map $S:C\tos C$ as
\[
S(y)=\{x\in C:~y\notin T(x)\}.
\]
Clearly, it satisfies $\cap_{x\in C}T(x)=\{y\in C:~S(y)=\emptyset\}$.

\begin{theorem}\label{KKM}
Let $C$ be a non-empty and convex subset of $\R^n$ and $T:C\tos C$ be a KKM map such that $S$ is lsc. If there exists $K$, a non-empty, convex and compact subset of $C$, such that $S(C)\subset K$; then
\[
 \bigcap_{x\in C}T(x)\neq\emptyset.
\] 
\end{theorem}
\begin{proof}
Due to Lemma \ref{lemma-RW}, the set-valued map $\co(S):C\tos C$, which is defined as
\[
\co(S)(x)=\co(S(x))
\]
is lsc too. If $S$ has non-empty values then, by  Corollary \ref{FPTLSC}, there exists $x\in C$ such that $x\in \co(S)(x)$. But this contradicts the fact that $T$ is a KKM map. Hence, there exists $y\in C$ such that $S(y)=\emptyset$.
\end{proof}

Notice that  Theorem \ref{KKM} can not be deduced from \cite[Lemma 1]{Kfan2}.
On the other hand, we can see that Corollary \ref{coro-MEP} is also a consequence of  Theorem \ref{KKM}.

\begin{remark}
It is important to notice that we can generalize the concept of projected solution if we introduce a set-valued map $P$ from $\R^n$ to $C$ with similar properties of the projection.
\end{remark}

\section{Applications}\label{applications}
In this section, as was mentioned earlier, we consider applications to the study of solutions of three particular problems. The first one is a special optimization problem known as a quasi-optimization problem, the second one is a quasi-variational inequality problem and finally,
 the third one is a generalized Nash equilibrium problem.
We will establish sufficient conditions to guarantee the existence of projected solutions for these problems, which were introduced
by Aussel, Sultana and Vetrivel in \cite{ASV-2016}.

\subsection{Quasi-optimization}
Given a real-valued function $h:\R^n\to \R$ and a set-valued map $K:C\tos C$, where
$C$ is a subset of $\R^n$, the {\em quasi-optimization problem} (QOpt) is described as
\begin{align}\label{QO}
\mbox{find }x_0\in K(x_0)\mbox{ such that }~~ x_0\in\argmin_{z\in K(x_0)} h(z).\tag{QOpt}
\end{align}

The terminology of quasi-optimization problem comes from \cite{GMP00} (see 
formula (8.3) and Proposition 12) and has been recently used in \cite{AC-2013,JC-JZ2,FKa10}. It 
emphasizes the fact that it is not a standard optimization problem since the 
constraint set depends on the solution, and it also highlights the parallelism 
to quasi-equilibrium problems.

As in \cite{ASV-2016}, a point $x_0\in C$ is said to be a projected solution of the 
QOpt if there exists $z_0\in \displaystyle\argmin_{z\in K(x_0)} h(z)$ such that $x_0\in P_C(x_0)$.

Using a reformulation as quasi-equilibrium problem, similar to the one in \cite{JC-JZ2},
we will characterize the projected solutions of \eqref{QO}. In that sense, associated to $h$, 
let us define the bifunction $f^h:\R^n\times \R^n\to\R$ as
\[
 f^h(x,y)=h(y)-h(x).
\]
Clearly, $f^h$ vanishes on the diagonal of $\R^n\times \R^n$. The following lemma follows from the definition of $f^h$.
\begin{lemma}\label{QEP-QOpt}
Let $x_0\in C$. 
Then, $x_0$ is a projected solution of \eqref{QO} if and only if, $x_0$ is a projected solution of 
\eqref{QEP} associated to $f^h$ and $K$.
\end{lemma}

We now show the existence of projected solutions for \eqref{QO} without the continuity nor the quasiconvexity of $h$, which is a consequence of Theorem~\ref{main-result-4} and it generalizes \cite[Theorem 4.1]{ASV-2016}.
\begin{theorem}
Let $C$ be a closed, convex and non-empty subset of $\R^n$, let $h:\R^n\to\R$ be a function and 
let $K:C\tos \R^n$ be a set-valued map such that $K(C)$ is relatively compact. 
If the following assumptions hold
\begin{enumerate}
\item $K$ is lsc with convex values;
\item the set $D=\{(x,z)\in C\times\R^n:~z\in K(x)\mbox{ and }x\in P_C(z)\}$ is closed;
\item the set-valued map $H:D\tos\R^n$ defined as
\[
H(x,z)=\{y\in K(x):~h(z)>h(y)\}
\] is lsc with convex values;
\end{enumerate}
then, there exists a projected solution of \eqref{QO}.
\end{theorem}

Notice that under continuity of $h$ and lower semicontinuity of $K$, the set-valued map $H$ is lsc too. However, the converse is not true in general as we can see in the following example.

\begin{example}
Let $h:\R\to\R$ be the function defined as
\[
h(x)=\left\lbrace\begin{array}{cl}
x,&x<1\\
2,&x=1\\
x+2,&x>1
\end{array}\right.
\] which is clearly quasiconvex but not continuous. 
  Now, consider the constraint map $K:C\tos C$ defined as $
K(x)=C=[0,2]$.
It is clear that $K$ is lsc with convex values. Moreover, it is not difficult to verify that the set-valued map $H:C\tos C$, defined by
\[
H(x)=\{y\in C:h(x)>h(y)\}=[0,x]
\]
is lsc.
\end{example}

We can see that under quasiconvexity of $h$, if $K$ is convex valued, then $H$ is  convex valued too. However, the following example shows us that the converse does not hold in general.

\begin{example}
Let $h:\R\to\R$ be a function  and let $K:[0,2]\tos [0,2]$ be a set-valued map both defined by 
\[
h(x)= \left\lbrace\begin{array}{cl}
                   |x-\frac{1}{2}|, & x\leq 1\\
                   |x-\frac{3}{2}|,& 1<x
                  \end{array}\right.
\quad \text{and} \quad  K(x)=\left\lbrace\begin{array}{cl}
                   \left[1/2,1\right]\cup\{x\},&0\leq x<1/2  \\
                   \left[x,1\right], &1/2\leq x\leq 1\\
                   \left[1,x\right],&1<x\leq 2 .
                  \end{array}\right. 
\]

\begin{figure}[h!]
\centering
\begin{tikzpicture}[scale=1.65]
\fill[color=white](0,0)--(2,0)--(2,-1)--(0,-1)--(0,0);
\draw[->] (-0.6,0)--(2.7,0)node[below]{$x$};
\draw[->] (0,-0.3)--(0,1.2)node[left]{$y$};
\draw[thick,<->](-0.5,1)--(0.5,0) -- (1,0.5) -- (1.5,0)--(2.5,1);
\draw(0.5,0)node[below]{$\displaystyle\frac{1}{2}$};
\draw(1.5,0)node[below]{$\displaystyle\frac{3}{2}$};
\draw(2,0)node[below]{$2$};
\draw(0,0.35)node[left]{$\displaystyle\frac{1}{2}$};
\draw[dashed](2,0)--(2,0.5);
\draw[dashed](0,0.5)--(2,0.5);
\end{tikzpicture}
\hfill
\begin{tikzpicture}[scale=1.5]
\draw[->] (-0.3,0)--(2.4,0)node[below]{$x$};
\draw[->] (0,-0.3)--(0,2.4)node[left]{$y$};
\fill[color=gray!50](0,0.5)--(0,1)--(1,1)--(0.5,0.5)--(0,0.5);
\draw(0,0.5)--(0,1)--(1,1)--(0.5,0.5)--(0,0.5);
\draw(0,0)--(0.5,0.5);
\fill[color=gray!50](1,1)--(2,2)--(2,1)--(1,1);
\draw(1,1)--(2,2)--(2,1)--(1,1);
\draw(1,-0.1)node[below]{$1$};
\draw(2,-0.1)node[below]{$2$};
\draw(0,1)node[left]{$1$};
\draw(0,2)node[left]{$2$};
\draw[dashed](1,0)--(1,1);
\draw[dashed](2,0)--(2,1);
\draw[dashed](0,2)--(2,2);
\end{tikzpicture}
\caption{Graphs of $h$ and $K$.} \label{graph}
\end{figure}
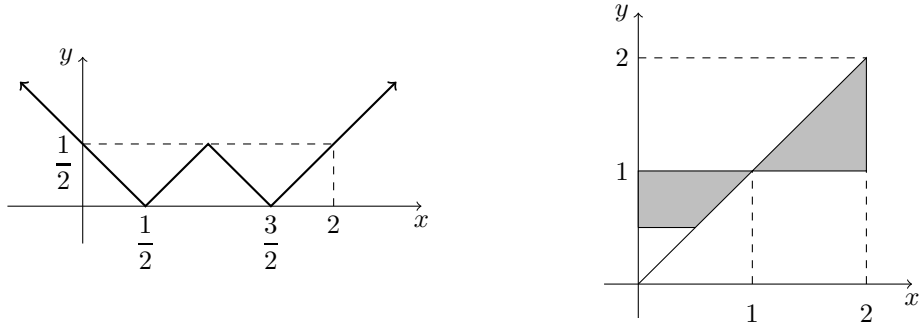

Figure \ref{graph} shows the graphs of $h$ and $K$. It is clear that the map $H:[0,2]\tos [0,2]$, 
defined as
\[
H(x)=\{y\in K(x):~h(y)<h(x)\}
\]
 has convex values. However, neither $h$ is quasiconvex nor $K$ is convex valued.
\end{example}

\subsection{Quasi-variational inequalities}
Let $C$ be a non-empty subset of $\R^n$ and $T,K$ two set-valued maps 
$T:\R^n\tos \R^n$ and $K:C\tos \R^n$. As in \cite{ASV-2016}, 
a point $x_0\in C$ is said to be a \emph{projected solution} of the quasi-variational inequality if, there exists $z_0\in X$ such that 
\begin{enumerate}
\item $x_0$ is a projection of $z_0$ onto $C$, $x_0\in P_C(z_0)$;
\item $z_0$ is a solution of the Stampacchia variational inequality associated to $T$ and $K(x_0)$, that is, $z_0\in K(x_0)$ and
\[
\mbox{there exists }z_0^*\in T(z_0)\mbox{ such that }\langle z_0^*, y-z_0\rangle\geq 0,~\text{for all } y\in K(x_0).
\]
\end{enumerate}

In a similar way, $x_0\in C$ is said to be a \emph{projected solution} of the Minty quasi-variational inequality if, there exists $z_0\in X$ such that
\begin{enumerate}
\item $x_0$ is the projection of $z_0$ onto $C$, $x_0\in P_C(z_0)$;
\item $z_0$ is a solution of the Minty variational inequality associated to $T$ and $K(x_0)$, that is, $z_0\in K(x_0)$ and
\[
\mbox{there exists }z_0^*\in T(z_0)\mbox{ such that }\langle y^*, y-z_0\rangle\geq 0,~\text{for all } y\in K(x_0).
\]
\end{enumerate}

It is clear that under pseudomonotonicity every projected solution of the Stampacchia variational inequality is a projected solution of the Minty  variational inequality. The converse holds if we assume the upper sign-continuity.


The following lemma establishes the relationship between the projected solution of a (Minty) quasi-variational inequality and the projected solution of a (Minty) quasi-equilibrium problem. The proof is a direct consequence of the definitions.

\begin{lemma}\label{Equiv-QEP-QVI}
Let $C$ be a subset of $\R^n$, $T:\R^n\tos \R^n$ be a set-valued map with compact values and
$K:C\tos \R^n$ be a set-valued map. 
Any projected solution of the (Minty) quasi-variational inequality associated to $T$ and $K$  is a projected solution of the (Minty) quasi-equilibrium problem associated to $f_T$ and $K$, where $f_T$ is defined as \eqref{T-f}. And conversely.
\end{lemma}

In \cite{ASV-2016}, the authors established two results on the existence of projected solutions for quasi-variational inequality problems on finite dimensional spaces. They used generalized monotonicity, specifically pseudomonotonicity and quasimonotonicity, for their proof. In a similar spirit,
we present another existence result using proper quasimonotonicity.
\begin{theorem}\label{QVI}
Let $T:\R^n\tos \R^n$ be a set-valued map with compact values, $C$ be a compact, convex and non-empty subset of $\R^n$, and $K:C\tos \R^n$  be a set-valued map.
If the following assumptions hold
\begin{enumerate}
 \item $K$ is closed and lsc with non-empty and convex values;
 \item $K(C)$ is compact;
 \item $T$ is properly quasimonotone;
 \item $\{(x,y)\in K(C)\times K(C):~\sup_{x^*\in T(x)}\langle x^*,y-x\rangle\leq0\}$ is closed;
\end{enumerate}
then, the Minty quasi-variational inequality admits at least one projected solution. Moreover, if $T$ is upper sign-continuous,
then the Stampacchia quasi-variational inequality has a projected solution.
\end{theorem}
\begin{proof}
Clearly, $f_T$ (defined as \eqref{T-f}) and $K$ satisfy the conditions of Theorem \ref{main-result-4}.
The result follows from Lemma \ref{Equiv-QEP-QVI}.
\end{proof}

As a direct consequence of the previous theorem, we have the following result on the existence of solutions to the quasi-variational inequality.
\begin{corollary}\label{E-QVI}
Let $T:\R^n\tos \R^n$ be a set-valued map with compact values, $C$ be a compact, convex and non-empty subset of $\R^n$, and $K:C\tos C$ be a set-valued map.
If the following assumptions hold
\begin{enumerate}
 \item $K$ is closed and lower semicontinuous with convex values;
 \item $T$ is properly quasimonotone;
 \item $T$ is upper sign-continuous;
 \item $\{(x,y)\in C\times C:~\sup_{x^*\in T(x)}\langle x^*,y-x\rangle\leq0\}$ is closed;
\end{enumerate}
then, the Stampacchia quasi-variational inequality admits at least one solution. 
\end{corollary}

\begin{remark} 
Here a few remarks are needed.
\begin{enumerate}
\item Theorem \ref{QVI} is also a consequence of Corollary \ref{P-result-GM} and hence, it works in Banach spaces, where $T$ could be considered with weak$^*$-compact values.

\item An analogous result to the previous corollary was proved in \cite[Proposition 3.5]{AC-2013}, where they did not require our assumption 4. Instead, they assumed the following technical condition: for all $x_\alpha\to x$ and all $y_\alpha\to y$
\begin{align}\label{technical}
\liminf \sup_{x_\alpha^*\in T(x_\alpha)}\langle x_\alpha^*,y_\alpha-x_\alpha\rangle\leq0~\Rightarrow~
\sup_{x^*\in T(x)}\langle x^*,y-x\rangle\leq0.
\end{align}
This condition implies assumption 4 in Corollary \ref{E-QVI}, but it is stronger. Indeed, consider for instance the set-valued map $T:\R\tos\R$, defined by
\[
T(x)=\left \lbrace\begin{array}{cl}
\{x\},&x\neq0\\
\{1\},&x=0
\end{array}\right.,
\]
which satisfies assumption 4 in Corollary \ref{E-QVI}. However, for $x_n=1/n$ and $y_n=1$, for all $n\in\N$, we have
\[
\liminf\langle x_n,y_n-x_n\rangle=0\mbox{ and }\langle 1,1-0\rangle>0.
\]
Thus, it fails to satisfy implication \eqref{technical}. 
\end{enumerate}
\end{remark}

Another result on the existence of projected solutions for quasi-variational inequality problems without generalized monotonicity is given below.

\begin{theorem}\label{QVI-2}
Let $C$ be a convex, compact and non-empty subset of $\R^n$ and consider the set-valued maps $K: C\tos \R^n$ and $T:\R^n\tos\R^n$. Assume that
\begin{enumerate}
\item $K$ is lsc with non-empty convex values;
\item the set $\{(x,z)\in C\times\R^n:~z\in K(x)\mbox{ and }x\in P_C(z)\}$ is closed;
\item the set $\{(x,y)\in \R^n\times \R^n:~\sup_{x^*\in T(x)}\left\langle x^*,y-x\right\rangle \geq0\}$ is closed.
\end{enumerate}
Then, the Stampacchia quasi-variational inequality has a projected solution.
\end{theorem}
\begin{proof}
Clearly, $f_T$ (defined as \eqref{T-f}) and $K$ satisfy the conditions of  Theorem \ref{Main-result-3}.
The result follows from  Lemma \ref{Equiv-QEP-QVI}.
\end{proof}

The following examples show that Theorem \ref{QVI} and Theorem \ref{QVI-2} are independent.
\begin{example}
Let $K:[-1,1]\tos\R$ be a set-valued map and $T:\R\to\R$ be a single-valued map defined as
\[
K(x)=\left[-3,0\right] \quad \mbox{and} \quad T(x)=\left\lbrace\begin{array}{cl}
-1,&x<0\\
1,&x\geq0
\end{array}\right..
\] 
Clearly, $K$ is closed and lsc with convex values, and $T$ is properly quasimonotone. 
It is not difficult to see that the set
\[
\{(x,y)\in \left[-3,0\right]\times \left[-3,0\right]:~\left\langle T(x),y-x\right\rangle\leq0 \}
\]
is closed. However the following set
\[
\{(x,y)\in \R\times \R:~\left\langle T(x),y-x\right\rangle\geq0 \}
\]
is not closed. Hence, we can use  Theorem \ref{QVI} in order to guarantee the existence of projected solutions.
\end{example}
\begin{example}
Let  $K:[0,1]\tos\R$ and
 $T:\R\tos\R$ be set-valued maps defined by
\[
K(x)= \left\lbrace \begin{array}{cl}
](x+1)/2,2],&x\in [0,1[\\
\left[1,2\right],&x=1
\end{array}\right. 
\]
 and $ T(x)=\{-1,1\}, \mbox{ for all }x\in \R$. Clearly, $T$ is not properly quasimonotone and
 $K$ is lsc with convex values, but it is not closed. On the other hand,
 \[
\{(x,y)\in \R\times \R:~\sup_{x^*\in T(x)}\left\langle x^*,y-x\right\rangle \geq0\}=\R\times\R
 \] 
 and, as the projection onto $C=[0,1]$ is single valued and continuous, we have
 \[
\{(x,z)\in [0,1]\times\R:~z\in K(x)\mbox{ and }x\in P_C(z)\}=\{1\}\times[1,2]. 
 \]
Therefore, the existence of projected solutions is due to Theorem \ref{QVI-2} and not Theorem \ref{QVI}. 
\end{example}

\subsection{GNEPs}
A generalized Nash equilibrium problem (GNEP) consists of
 $p$ players. Each player $\nu$ controls the decision variable $x^\nu\in C_\nu$, where $C_\nu$ is a non-empty convex and closed subset
 of $\R^{n_\nu}$. We denote by $x=(x^1,\dots,x^p)\in \prod_{\nu=1}^p C_\nu=C$ the vector formed by all these decision
 variables and by $x^{-\nu}$, we denote the strategy vector of all the players different from player $\nu$. The set of all such
vectors will be denoted by $C^{-\nu}$. We sometimes write  $(x^\nu,x^{-\nu})$ instead of $x$ in order to emphasize the $\nu$-th player's
variables within $x$. Note that this is still the vector $x=(x^1,\dots,x^\nu,\dots,x^p)$, and the notation $(x^\nu,x^{-\nu})$
does not mean that the block components of $x$ are reordered in such a way that $x^\nu$ becomes the first block.
Each player $\nu$ has an objective function $\theta_\nu:C\to\R$ that depends on all player's strategies. Each player's strategy must belong to a set identified by the set-valued map $K_\nu:C^{-\nu}\tos C_\nu$ in the sense that the strategy
space of player $\nu$ is $K_\nu(x^{-\nu})$, which depends on the rival player's strategies $x^{-\nu}$.
Given the strategy $x^{-\nu}$, player $\nu$ chooses a strategy $x^\nu$ such that it solves the following optimization problem
\begin{align}\label{opt-nu}
\min_{x^\nu} \theta_\nu(x^\nu,x^{-\nu}),\mbox{ subject to }x^\nu\in K_\nu(x^{-\nu}),
\end{align}
for any given strategy vector $x^{-\nu}$ of the rival players. The solution set of problem \eqref{opt-nu}
is denoted by ${\rm Sol}_\nu(x^{-\nu})$. Thus, a {\em generalized Nash equilibrium} is a vector 
$\hat{x}$ such that $\hat{x}^\nu\in {\rm Sol}_\nu(\hat{x}^{-\nu})$, for any $\nu$.\\
\par

Suppose that $n$ and $n_\nu$ are natural numbers that satisfy $n= \sum_{\nu=1}^p n_\nu$. For any $\nu\in\{1,2,...,p\}$, let $C_\nu$ be a non-empty subset of $\R^{n_\nu}$, and let 
$K_\nu:C^{-\nu}\tos \R^{n_\nu}$ and
$\theta_\nu:\R^n\to\R$, be set-valued maps. As in \cite{ASV-2016}, a vector $\hat{x}$ of $C$ is said to be a {\em projected solution} of the generalized Nash equilibrium problem 
if, there exists $\hat{z}\in \R^n$ such that:
\begin{enumerate}
 \item $\hat{x}$ is a projection of $\hat{z}$ onto $C$;
\item  $\hat{z}$ is a solution of the Nash equilibrium problem defined by all functions $\theta_\nu$, where $\nu \in \{1,2,...,p\}$,
and the constraint sets $K_\nu(\hat{x})$, $\nu \in \{1,2,...,p\}$, that is, for any $\nu$, $\hat{z}^\nu\in K_\nu(\hat{x}^{-\nu})$
is a solution of the following optimization problem
\begin{align}\label{NEP}
\min_{z^\nu} \theta_\nu(z^\nu,\hat{z}^{-\nu}),\mbox{ subject to }z^\nu\in K_\nu(\hat{x}^{-\nu}) .
\end{align}
\end{enumerate}

Associated to a GNEP, there is a bifunction $f^{NI}:\R^n\times \R^n\to\R$, defined by
\[
 f^{NI}(x,y)=\sum_{\nu=1}^p\{\theta_\nu(y^\nu,x^{-\nu})-\theta_\nu(x^\nu,x^{-\nu})\},
\]
which is called Nikaid\^o-Isoda bifunction and was introduced in \cite{Nikaido-Isoda}. This bifunction has a simple interpretation. Suppose that $x$ and $y$ are two feasible points
for the GNEP. Each summand in the definition represents the improvement in the objective function of player $\nu$
when he changes his action from $x^\nu$ to $y^\nu$, while all the other players stick to the choice $x^{-\nu}$.

Additionally, we define the set-valued map $K:C\tos \R^n$ as 
\[
K(x)=\prod_{\nu=1}^p K_\nu(x^{-\nu}).
\]
Now, we can characterize all projected solution of a GNEP in the following result. 

\begin{lemma}\label{equiv-2}
A vector $\hat{x}$ is a projected solution of a GNEP if and only if, it is a projected solution of \eqref{QEP} associated to $f^{NI}$ and
$K$.
\end{lemma}
\begin{proof}
Clearly, every projected solution of a GNEP is a projected solution of \eqref{QEP} associated to $f^{NI}$ and $K$. Conversely, if $\hat{x}\in C$ and there exists $\hat{z}\in K(x)$ such that $\hat{x}\in P_C(\hat{z})$, and
\begin{align}\label{ine-N}
 f^{NI}(\hat{z},y)\geq0,~\mbox{ for all }y\in K(\hat{x}),
\end{align}
then, for all $z^\nu\in K_\nu(\hat{x}^{-\nu})$, the vector $(z^\nu,\hat{z}^{-\nu})\in K(\hat{x})$ and the inequality in \eqref{ine-N}
becomes
\begin{align*}
\theta_\nu(\hat{z}^\nu,\hat{z}^{-\nu})\leq \theta_\nu(z^\nu,\hat{z}^{-\nu}),
\end{align*}
that is, $\hat{z}^\nu$ is a solution of \eqref{NEP}.
\end{proof}

Thanks to  Lemma \ref{equiv-2}, we can deduce from Theorem \ref{P-result} the following result on the existence of projected solutions of a GNEP, which generalizes \cite[Theorem 4.2]{ASV-2016}.

\begin{theorem}\label{FGNEP}
For any $\nu \in \{1,2, ...,p \}$, let $C_\nu$ be a non-empty, closed and convex subset of $\R^{n_\nu}$, 
$\theta_\nu:\R^n\to\R$ be a function and $K_\nu:C^{-\nu}\tos \R^{n-n_\nu}$ be a set-valued map. Then, the GNEP
admits a projected solution if
\begin{enumerate}
\item for each $\nu$, $\theta_\nu$ is continuous and convex with respect to the $x^{\nu}$ variable;
\item for each $\nu$, the map $K_\nu$ is lsc with non-empty and convex values, and $K_\nu(C)$ is relatively compact;
\item the set $D=\{(x,z)\in C\times\R^n:~x\in P_C(z)\mbox{ and }z\in K(x)\}$ is closed.
\end{enumerate}
\end{theorem}
\begin{proof}
It is enough to see that the set-valued map $R:D\tos\R^n$, defined as
\[
R(x,z)=\{y\in K(x):~f^{NI}(z,y)<0\}
\]
is lsc. The result follows from Theorem \ref{Main-result-3}.
\end{proof}
\begin{remark}
If in the previous theorem we add the closeness of the graph of any constraint map, then assumption 3 holds. Moreover, it is true in Banach spaces. 
The existence results for generalized Nash equilibrium problems for infinite dimensional spaces 
has been recently studied, see for instance \cite{castellani_Giuli15,AGM-16}
\end{remark}
We finish this subsection with
the following result, which is a consequence of Theorem \ref{FGNEP}.
\begin{corollary}\cite[Theorem 2.2]{Cubio}
For any $\nu \in \{1,2, ...,p \}$, let $C_\nu$ be a non-empty, compact and convex subset of $\R^{n_\nu}$, $\theta_\nu:\R^n\to\R$ be a function and $K_\nu:C^{-\nu}\tos C_\nu$ be a set-valued map. Then, the GNEP admits a  solution if
\begin{enumerate}
\item for each $\nu$, $\theta_\nu$ is continuous and convex with respect to the $x^{\nu}$ variable;
\item for each $\nu$, the map $K_\nu$ is lsc with non-empty and convex values;
\item $\Fix(K)$ is closed.
\end{enumerate}
\end{corollary}

\bibliographystyle{abbrv}

\end{document}